\documentclass[11pt,reqno]{amsart}
\usepackage{amssymb}
\usepackage{amsmath}
\usepackage[active]{srcltx}
\usepackage{t1enc}
\usepackage[latin2]{inputenc}
\usepackage{verbatim}
\usepackage{amsmath,amsfonts,amssymb,amsthm}
\usepackage[mathcal]{eucal}
\usepackage{enumerate}
\usepackage[centertags]{amsmath}
\usepackage{graphics}
\usepackage{xcolor}

\newtheorem*{thm}{Theorem A } 
\newtheorem*{MenshovRademacher}{Theorem B}
\newtheorem*{remark}{Remark}

\newtheorem{theorem}{Theorem}

\newtheorem{lemma}{Lemma}

\newtheorem{corollary}{Corollary}

\begin{document}
	\author{L-E. Persson, V. Tsagareishvili and G. Tutberidze}
	\title[Properties of general Fourier series]{Properties of sequence of linear functionals on $BV$ with applications}
	\address{L-E. Persson, UiT-The Arctic University of Norway, P.O. Box 385, N-8505, Narvik, Norway and Department of Mathematics and Computer Science, Karlstad University, 65188 Karlstad, Sweden}
	\email {lars.e.persson@uit.no \and larserik6pers@gmail.com}  
	\address{ V. Tsagareishvili, Department of Mathematics, Faculty of Exact and Natural Sciences, Ivane Javakhishvili Tbilisi State University, Chavchavadze str. 1, Tbilisi 0128, Georgia }
	\email{cagare@ymail.com}
	\address{G.Tutberidze, The University of Georgia, School of science and technology, 77a Merab Kostava St, Tbilisi 0128, Georgia}
	\email{g.tutberidze@ug.edu.ge \and giorgi.tutberidze1991@gmail.com}
	
	\thanks{}
	\date{}
	\maketitle

	\begin{abstract}
		This paper is devoted to investigating the sequence of some linear functionals in the space $BV$ of finite variation functions. We prove that under certain conditions this sequence is bounded. We also prove that this result is sharp. In particular, the obtained results can be used to study convergence of some general Fourier series. Moreover, the obtained conditions seem to be new and useful also for classical orthonormal systems.
	\end{abstract}
	
	

\section{INTRODUCTION}
In order not to disturb the discussion in this introduction and the proofs of our main result we have collected all notations, definitions and other preliminaries in Section 2.

In this paper we prove a new convergence result for a special sequence of linear functionals $\{U_n\}=\{U_n(f)\},$ defined by \eqref{equation2}-\eqref{A1} and where usually $f \in BV$ on $(0,1).$ See Theorem \ref{theorem3.1}. We also prove that this result is, in a sense, sharp. See Theorem \ref{theorem3.2}.

The study of functionals has a rich history and many powerful and interesting results are obtained, see e.g. the monographs \cite{AkhiezerGlazman, Banach, DunfordScshvartz, Edvards, Kantorovich, KolmogorovFomin, Yosida} and the references therein. And this interest seems only to increase and one reason is absolutely the fact that such developments are powerful for various applications.

For our investigation it is important to remind about the fact  that from Banach's Theorem (see e.g.\cite{Banach}) it follows that if  $f \in L_2 (0,1),\ \left(f\nsim 0\right),$ then there exists an ONS such that  the  Fourier series of  this function $f$ is not convergent on $[0,1]$ with respect to this system. Thus, it is clear that the Fourier coefficients of functions of bounded variation do not in general satisfy  condition in Theorem B (the Menchov-Rademacher Theorem). 

Another motivation for this paper is to use our main result to obtain some new results concerning convergence/divergence of general Fourier series. Some other results for this case can be found in \cite{GogoladzeTsagareishvili1, GogoladzeTsagareishvili2, GogoladzeTsagareishvili3, GogoladzeTsagareishvili, Tsagareishvili1, Mclaughlin, PTT, PTW, PTW1, PSTW, tep1, tep2, tep4, Tsagareishvili2, Tsagareishvili3, Tsagareishvili4, Tsagareishvili0, tsatut1, tsatut2,tsatut3}. See also the monograph \cite{KashinSaakyan}.

The main results Theorems \ref{theorem3.1} and \ref{theorem3.2} are presented and proved in Section 3. The new applications concerning convergence of general Fourier series can be found Section 4. See Theorem 3, Corollary 1, Theorem 4 and Theorem 5.

\section{PRELIMINARIES}
Let $\{\varphi_n\}$ be an orthonormal system (ONS) on $\left[0,1\right].$

We denote by $BV$ the class of all functions of bounded variation on $(0,1)$ and write $V \left(f\right)$ for the total variation of a function $f$ on $\left[0,1\right]$.

By $A$ we denote the Banach space of absolutely convergent functions with the norm  ${{\left\| f \right\|}_{A}}$ defined by

\begin{equation} \label{*}
	{{\left\| f \right\|}_{A}}\,:=\,{{\left\| f \right\|}_{C}}\,+\,\int\limits_{0}^{1}{\left| \frac{df}{dx} \right|dx.}
\end{equation}
We will investigate the linear functionals $\{U_n(f)\}$ defined by
\begin{equation}\label{equation2}
	{U}_{n}(f):=\int_{0}^{1} f(x) Q_n(d,a,x) dx,
\end{equation}
where $f\in L_2,$ $a=\{a_n\} \in l_2$ is an arbitrary sequence of numbers and

\begin{equation}\label{equation3}
	{{Q}_{n}}(d,a,x):=\sum\limits_{k=1}^{n}{{{d}_{k}}{{a}_{k}}{\log k}{{\varphi }_{k}}(x).}
\end{equation}
Here $\{d_n\}$ denote a sequence of real number such that
\begin{eqnarray}
	\label{A1} d_n=O\left(\frac{\sqrt{n}}{\log^2 (n+1)}\right).
\end{eqnarray}

For this investigation of the functionals $\{{U}_{n}(f)\}$ we need the following important Lemma (see \cite{GogoladzeTsagareishvili}).

\begin{lemma}\label{lemma1.1} If $f \in L_2 \left(0,1\right)$ takes only finite values on $\left[0,1\right]$ and $h \in L_2 \left(0,1\right)$  is an arbitrary function, then
	
	\begin{eqnarray} \label{1.2}
		\int_{0}^{1} f\left(x\right) h \left(x\right) dx 
		&=& \sum_{i=1}^{n-1}\left(f\left(\frac{i}{n}\right)-f\left(\frac{i+1}{n}\right)\right)\int_{0}^{i/n}h\left(x\right)dx \\
		&+& \sum_{i=1}^{n}\int_{\left(i-1\right)/n}^{i/n}\left(f\left(x\right) -f\left(\frac{i}{n}\right)\right)h\left(x\right)dx   \notag \\
		&+&f\left(1\right)\int_{0}^{1}h\left(x\right)dx.     \notag
	\end{eqnarray}
\end{lemma}

We denote 
\begin{eqnarray} \label{1.4}
	B_n\left(d,a\right)=\max_{1\leq i < n} \left\vert \int_{0}^{i/n} Q_n\left(d,a,x\right)dx\right\vert.
\end{eqnarray}

We say that the sequence of functionals  $\{U_n(f)\}$ is bounded on the space $V,$ if, for any $\{a_n\} \in l_2,$ 
$$\underset{n\to \infty }{\mathop{\limsup}} \left| {{U}_{n}}(f) \right|<+\infty. $$

We also need the following result of S.Banach (see e.g. \cite{Banach}):

\begin{thm} 
	Let $f\in L_2$ be an arbytrary (non-zero) function. Then there exists an ONS $\{{\varphi}_{n}\} $ such that
	$$\limsup_{n \rightarrow \infty}\left|\sum_{k=1}^n C_k(f) \varphi_k(x)\right|=+\infty \ \ \text{a.e. on} \ \ [0,1],$$
	where $C_k(f)$ are the Fourier coefficients of the function $f \in L_2$ with respect to the system $\{{\varphi }_{k}\} $ and	defined as follows  
	\begin{equation} \label{fourier}
		{{C}_{k}}(f):=\int\limits_{0}^{1}{f(x){{\varphi }_{k}}(x)dx.}
	\end{equation}
	\end{thm}
Moreover, we recall the following well-known result of Menshov and Rademacher (see e.g. \cite{KashinSaakyan} Ch.9, p.332).
\begin{MenshovRademacher}
	 \label{theorem1.1} 
	If $\{\varphi_n\}$ is an ONS on $\left[0,1\right]$ and a number sequence $\{c_n\}$ satisfies the condition
	\begin{eqnarray}
		\sum_{n=1}^{\infty}c_n^2 \log_2^2 n < +\infty, \notag
	\end{eqnarray}
	then the series 
	\begin{eqnarray} \notag
		\sum_{n=1}^{\infty}c_n \varphi_n\left(x\right)
	\end{eqnarray}
	converges a.e. on $\left[0,1\right]$.
	
\end{MenshovRademacher}

\section{The Main Results}
Our first main result reads:
\begin{theorem}\label{theorem3.1}
If, for any $\{a_n\}\in l_2,$ 
	\begin{eqnarray} \label{3.1}
		\qquad	B_n\left(d,a\right)=O\left(1\right),
	\end{eqnarray}
 then the sequence of functionals $\{U_n(f)\}$ is bounded on the space $BV$ for every $f\in BV.$ 

\end{theorem}

\begin{proof}

	By using Lemma \ref{lemma1.1}, when $h\left(x\right) = Q_n \left(d,a,x\right)$ we have
	\begin{eqnarray} \label{key8}
		\int_{0}^{1} f(x) Q_n(d,a,x)dx&=&\sum_{i=1}^{n-1}\left(f\left(\frac{i}{n}\right)-f\left(\frac{i+1}{n}\right)\right)\int_{0}^{i/n}Q_n(d,a,x)dx \notag \\
		&&+ \sum_{i=1}^{n}\int_{\left(i-1\right)/n}^{i/n}\left(f\left(x\right) -f\left(\frac{i}{n}\right)\right)Q_n(d,a,x)dx  \notag  \\
		&&+f\left(1\right)\int_{0}^{1}Q_n(d,a,x)dx:=A_1+A_2+A_3.  
	\end{eqnarray}

Let $f\in BV$, then we get (see \eqref{1.4})

\begin{eqnarray*}
	\left|A_1\right| &&\leq \max_{1\leq i < n} \left|\int_{0}^{i/n} Q_n(d,a,x)dx\right| \  \cdot \  \sum_{i=1}^{n-1}\left|f\left(\frac{i}{n}\right)-f\left(\frac{i+1}{n}\right)\right| \\
	\notag
	\\
	&&\leq V(f) B_n(d,a).
\end{eqnarray*}
Hence, from condition \eqref{3.1} it follows that
$$\left|A_1\right|=O(1)V(f).$$

By applying Hölder's inequality and \eqref{A1} , we get (since $\{a_n\}\in l_2$)
\begin{eqnarray*}
		\left|A_2\right| && \leq \sum_{i=1}^n \sup_{x\in \left[\frac{i-1}{n}, \frac{i}{n}\right]} \left|f(x)-f\left(\frac{i}{n}\right)\right| \int_{\left(i-1\right)/n}^{i/n}\left|Q_n(d,a,x)\right|dx  \\ &&\leq V(f) \max_{1\leq i \leq n} \int_{\left(i-1\right)/n}^{i/n}\left|Q_n(d,a,x)\right|dx \\ &&\leq   V(f) \frac{1}{\sqrt{n}} \left( \int_{0}^{1} Q_n^2(d,a,x)dx \right)^{1/2} \\
		&& =   \frac{ V(f)}{\sqrt{n}} \left( \int_{0}^{1} \left(\sum_{k=1}^{n}d_k a_k \log k \varphi_k(x)\right)^2 dx \right)^{1/2} \\
		&& = \frac{ V(f)}{\sqrt{n}} \left( \sum_{k=1}^{n} d_k^2 a_k^2 \log^2 k  \right)^{1/2} \\ &&=  V(f) \cdot \frac{\max\limits_{1\leq k \leq n} |d_k| } {\sqrt{n}} \cdot \log n \left(\sum_{k=1} ^n a_k^2\right)^{1/2} = O(1)V(f).
	\end{eqnarray*}
	
	By using \eqref{1.4} and Cauchy's inequality, for any $\{a_n\}\in l_2$ we find that
	\begin{eqnarray*}
		\left|A_3\right| &&= O(1) \left|\int_{0}^{1} Q_n(d,a,x)dx\right| \\ &&\leq O(1)\left( \left|\int_{0}^{1-1/n} Q_n(d,a,x)dx\right|+ \left|\int_{1-1/n}^{1} Q_n(d,a,x)dx\right|\right) \\
		&& \leq O(1) \left(B_n(d,a) + \frac{1}{\sqrt{n}} \left(\int_{0}^{1} Q_n^2 (d,a,x) dx\right)^{1/2}\right) = O(1).
	\end{eqnarray*}
Taking into consideration in \eqref{key8} the above estimates of $|A_1|,$  $|A_2|$  and $|A_3|$ we have that
	$$\left|\int_{0}^{1} f(x)\sum_{k=1}^{n}d_k a_k \log k \varphi_k(x)\right| =O(1).$$
	
It follows that
	\begin{equation} \label{equation6}
		\left|U_n(f) \right|\leq M(f),
	\end{equation}
	where $M(f)$ is a constant which does not depend on $n$ and the proof is complete.
	
\end{proof}

	Next we state a result which, in particular, show that the statement in Theorem \ref{theorem3.1} is, in a sense, sharp.
	
\begin{theorem}\label{theorem3.2}
	If for some $\{b_n\}\in l_2$
	$$\limsup_{n \rightarrow \infty}  B_n(d,b)=+\infty,$$
	then there exists a function $g\in A$, such that
	$$\limsup_{n \rightarrow \infty}  \left|U_n(g)\right|=+\infty.$$
\end{theorem}
\begin{proof}
	First we suppose that
	\begin{equation*}
	\limsup_{n \rightarrow \infty} \left|\int_{0}^{1} Q_n\left(d,b,x\right)dx\right|=+\infty.
	\end{equation*}

Then, if  $g(x)=1,\text{ }x\in[0,1],$ we have 
\begin{equation*}
	\limsup_{n \rightarrow \infty} \left|\int_{0}^{1} g(x)Q_n\left(d,b,x\right)dx\right|=+\infty.
\end{equation*}
Obviously $g \in A.$ Theorem \ref{theorem3.2} holds in this case.
	
	Next  we suppose that
	\begin{eqnarray}\label{equation7}
		\left|\int_{0}^{1}Q_n\left(d,b,x\right)dx\right|=O(1).
	\end{eqnarray}
Let $1\leq i_n < n$ be an integer, such that 
	$$B_n\left(d,b\right)=\max_{1\leq i < n} \left|\int_{0}^{i/n}Q_n\left(d,b,x\right)dx\right|=\left|\int_{0}^{i_n/n}Q_n\left(d,c,x\right)dx\right|.$$
	
	Suppose that for some sequence $b=\{b_k\}\in l_2$
	\begin{eqnarray}
		\label{3.6}    \limsup_{n \rightarrow \infty}  B_n(d,b)=+\infty.
	\end{eqnarray}
Consider the following sequence of test functions:
	\begin{equation*}
		f_{n}\left( x\right) =\left\{ 
		\begin{array}{ccc}
			0, & \text{when} & x\in \left[ 0,\frac{i_{n}}{n}\right] \\ 
			1, & \text{when} & x\in \left[ \frac{i_{n}+1}{n},1\right] \\ 
			\text{continuous and linear, } & \text{when} & x\in \left[ \frac{i_{n}}{n}, \frac{i_{n}+1}{n} \right].
		\end{array}%
		\right.
	\end{equation*}
Then (see \eqref{*}) 
	$$\left\Vert f_n\right\Vert_A = \int_{0}^{1}\left|f_{n}^{'} \left(x\right)\right|dx+\left \Vert f_n\left(x\right) \right \Vert_C\leq 2.$$
Furthermore,
	\begin{eqnarray}
		\label{3.7} && \left|\sum_{i=1}^{n-1} \left(f_n\left(\frac{i}{n}\right)-f_n\left(\frac{i+1}{n}\right)\right) \int_{0}^{i/n}Q_n\left(d,b,x\right)dx\right| \\
		&&=\left|\int_{0}^{{i_n}/n}Q_n\left(d,b,x\right)dx\right|=B_n\left(d,b\right). \notag
	\end{eqnarray}
Then, if $x\in \left[\frac{i-1}{n}, \frac{i}{n}\right]$ we find that
	\begin{equation*}
		\left|f_{n}\left( x\right)-f_{n}\left( \frac{i}{n}\right)\right| \left\{ 
		\begin{array}{ccc}
			\leq 1, & \text{if} & i=i_{n}+1, \\ 
			0, & \text{if} & i \neq i_{n}+1, \\ 
		\end{array}%
		\right.
	\end{equation*}
	and it implies that  (since $\{b_n\} \in l_2$)
	\begin{eqnarray}
		\label{3.8}  &&\left|\sum_{i=1}^{n} \int_{\left(i-1\right)/n}^{i/n} \left(f\left(x\right)-f\left(\frac{i}{n}\right) \right) Q_n\left(d,b,x\right)dx\right| \\
		&& \leq \sum_{i=1}^n  \sup_{x\in \left[\frac{i-1}{n}, \frac{i}{n}\right]} \left|f\left(x\right)-f\left(\frac{i}{n}\right)\right|\int_{\left(i-1\right)/n}^{i/n}\left| Q_n(d,b,x) \right| dx\notag \\
		&&\le V(f)\underset{1\le i\le n}{\mathop{\max }}\,\int\limits_{{(i-1)}/{n}}^{{i}/{n}}{ \left|Q_n(d,b,x)\right| dx} \notag\\ 
		&& =O(1)\frac{1}{\sqrt{n}}\left(\int_{0}^1  Q_n^2(d,b,x) dx \right)^{1/2} \notag \\
		&&=O(1)\frac{1}{\sqrt{n}}\max_{1\leq k \leq n }d_k \log n{{\left( \sum\limits_{k=1}^{n}{b_{k}^{2}} \right)}^{\frac{1}{2}}}=O(1).	\notag	
	\end{eqnarray}

	Consequently, by using \eqref{1.2} when $f\left(x\right)=f_n\left(x\right)$ and $Q_n\left(d,a,x\right)=Q_n\left(d,b,x\right),$ and combining \eqref{equation7}, (\ref{3.7}), (\ref{3.8}), we get that
	$$\left|\int_{0}^{1}f_n\left(x\right)Q_n\left(d,b,x\right)\right|dx \geq B_n(d,b)-O\left(1\right)-O(1).$$
	
	From here and from (\ref{3.6}) we have, that
	$$\limsup_{n\rightarrow \infty}{\left|\int_{0}^{1}f_n\left(x\right)Q_n\left(d,b,x\right)dx\right|}=+\infty.$$
	
	Finally, we note that since 
	$$U_n\left(f\right)=\int_{0}^{1}f\left(x\right)Q_n\left(d,b,x\right)dx$$
	is a sequence of linear bounded functionals on $A$, then by the Banach-Steinhaus theorem, there exists a function $g\in A$ such that
	\begin{eqnarray}\label{equation11}
		\limsup_{n\rightarrow \infty}|U_n(g)|=\limsup_{n\rightarrow \infty}{\left|\int_{0}^{1}g\left(x\right)Q_n\left(d,b,x\right)dx\right|}=+\infty.
	\end{eqnarray}

	The proof is complete.
	
\end{proof}

\section{ Applications concerning convergence of general Fourier series}

Our first application reads:
\begin{theorem}
	\label{theorem3.3}
	If condition \eqref{3.1} of Theorem \ref{theorem3.1} holds then, for any function $f\in BV,$
	$$\sum_{k=1}^{\infty} d_k^2 C_k^2(f)\log^2 k<+\infty .$$
	
\end{theorem}

	\begin{proof}
	By using condition \eqref{3.1} of Theorem \ref{theorem3.1}, and using equation \eqref{equation6} and \eqref{fourier} we have that
	\begin{eqnarray*}
	\sum_{k=1}^{n} d_k C_k(f)\log k &&= \int_{0}^{1} f(x) 	\sum_{k=1}^{n} d_k \log k \varphi_{k}(x) dx \\ &&=\int_{0}^{1} f(x) Q_n(d,a,x) dx.
		\end{eqnarray*}
	Hence,
		$$\sum_{k=1}^{n} d_k a_k C_k(f) \log k = U_n(f). $$
		Since
		$$\left|U_n(f)\right|=O(1)$$
		it follows that
		$$\sum_{k=1}^{n} d_k a_k \log k C_k(f)= O(1). $$
	
	Now, if we suppose that for any $\{a_k\} \in l_2,$
	$$\underset{n\to \infty }{\mathop{\limsup }}\,\,\left| {{U}_{n}}(f) \right|<+\infty ,$$
	then the following series
		\begin{equation}\label{equation12}
			\sum_{k=1}^{\infty} d_k a_k \log k C_k(f)  
		\end{equation}
	is convergent.
	
	This means that  $\{d_k C_k(f) \log k\}  \in l_2,$ or
	$$\sum_{k=1}^{\infty} d_k^2 C_k^2(f)\log^2 k<+\infty. $$
	The proof is complete.
		
	\end{proof}
In particular, Theorem A and Theorem \ref{theorem3.1} imply the following new result:

\begin{corollary}	\label{theorem3.4}
	If condition \eqref{3.1} of Theorem \ref{theorem3.1} holds  for any function $f\in BV,$ then the following series
	$$\sum_{k=1}^{\infty} d_k C_k (f) \varphi_{k}(x) $$
	is convergent a.e. on $[0,1].$
	
\end{corollary}

\begin{remark}
If condition \eqref{3.1} of Theorem \ref{theorem1.1} is fulfilled and  $d_k=1,$ $k = 1, 2, \dots , $ then, for any $f\in BV$ the series
$$\sum_{k=1}^{\infty} C_k (f) \varphi_{k}(x) $$
is convergent a.e. on $[0,1].$
\end{remark}

Next, we state the following result showing that Theorem \ref{theorem3.3} is, in  a sense, sharp.

\begin{theorem}
	\label{theorem3.5}
	For any function $g\in A$ $(g \ne 0 )$ there exists an ONS  $\{\varphi_{n}\}$ such that for some $\{a_n\} \in l_2$ and $d_k=1,$ $k = 1, 2, \dots  $
	\begin{eqnarray}
		\label{eq*} \limsup_{n \rightarrow \infty}\sum_{k=1}^{n}C_k^2(f) \log ^2 k=\limsup_{n \rightarrow \infty}\left|U_n(g)\right|=+\infty.
	\end{eqnarray}
\end{theorem}
\begin{proof}
	Let $g$ be an arbitrary function. According to the Banach Theorem there exists an ONS $\{\varphi_{n}\}$ such that
	\begin{eqnarray}
		\label{14} 
		\limsup_{n \rightarrow \infty}\left|\sum_{k=1}^{n} C_k(g) \varphi_{k}(x) \right|=+\infty \ \  \text{a.e. on} \ [0,1].
	\end{eqnarray}
Consequently, by using \eqref{14} and Theorem A,  we conclude that

\begin{eqnarray}
	\label{15} 
	\sum_{k=1}^{\infty} C_k^2(g)\log^2 k=+\infty.
\end{eqnarray}

	Indeed, suppose the contrary to \eqref{eq*} namely that for arbitrary $\{a_n\} \in l_2$
	$$\limsup_{n \rightarrow \infty}\left|U_n(g)\right|<+\infty.$$
	Then as it follows from \eqref{equation12} when $d_k=1,$ $k = 1, 2, \dots  $ for any  $\{a_n\} \in l_2$ that the series 
	$$\sum_{k=1}^{\infty} a_k\log k C_k(g)$$
	is convergent. Thus, $\{C_k(g) \log k\} \in l_2$ or
	$$\sum_{k=1}^{\infty}  C_k^2(g) \log^2 k <+\infty,$$
	which contradicts \eqref{15}.
	
	This contradiction shows that \eqref{eq*} holds so the proof is complete.
	
\end{proof}

Finally, we state the following efficiency result:
\begin{theorem}
	\label{theorem4.1}
	Let $\{\varphi_{n}\}$ be an ONS such that uniformly with respect to $x\in[0,1]$ it holds that
	\begin{eqnarray}
		\label{16}
		\int_{0}^{x} \varphi_{n}(x)dx = O \left(\frac{1}{n}\right).
	\end{eqnarray}
Then for any $a=\{a_n\} \in l_2,$ 
\begin{eqnarray}
	\label{**} B_n(d,a) \leq  \max_{x\in [0,1]} \left|\int_{0}^{x} \sum_{k=1}^{n} d_k a_k \log k \varphi_{k}(u) du \right|= O(1).
\end{eqnarray}
\end{theorem}

\begin{proof}
	According to \eqref{16} and by the Cauchy inequality we get (see \eqref{A1})
	\begin{eqnarray*}
		B_n(d,a) && \leq  \max_{x\in [0,1]} \left|\int_{0}^{x} \sum_{k=1}^{n} d_k a_k \log k \varphi_{k}(u) du \right| \\
		&& =\max_{x\in [0,1]} \left| \sum_{k=1}^{n} d_k a_k \log k \int_{0}^{x}\varphi_{k}(u) du \right| = O(1) \left| \sum_{k=1}^{n} a_k d_k \log k \frac{1}{k}\right| \\
		&& = O(1) \left(\sum_{k=1}^{n} a_k^2\right)^{1/2} \left(\sum_{k=1}^{n} d_k^2 \log^2 k \frac{1}{k^2}\right)^{1/2} \\
		&& =O(1) \left(\sum_{k=1}^{n} \frac{k}{\log^4 (k+1)} \frac{\log^2 k}{k^2}\right)^{1/2} =O(1).
	\end{eqnarray*}

Hence, \eqref{**} is proved so the proof is complete.
 
\end{proof}

\textbf{Remark:} Consequently, the functionals defined by \eqref{equation2} are bounded e.g. when $\{\varphi_{n}\}$ is the trigonometric or Walsh system. \\

\textbf{Final remark.} We pronounce that some other convergence/divergence result  of one-dimensional Vilinkin-Fourier series can be found in the new book \cite{PTW1}. We hope that our (Functional) approach can be used to contribute to solving some of the open questions raised in this book. In this connection  we also mention the new paper \cite{PSTW}, which is related to the famous Carleson paper \cite{Ca}.


\begin{thebibliography}{99}
	\bibitem{AkhiezerGlazman} \textbf{N.I. Akhiezer and I.M. Glazman,} Theory of liner operators in Hilbert Spaces, Dover Publishers, 1993, 400 pp.

	\bibitem{Banach} \textbf{S. Banach,} Theorie des operations lineaires, American Mathematical Society 1978, 259 pp. 
	
	\bibitem{Ca} \textbf{L. Carleson,} On convergence and growth of partial sums of Fourier series, Acta Math 116 (1966), no.1, 135-157.
	
	\bibitem{DunfordScshvartz} \textbf{N. Dunford and J. schwartz,} Liner operators, part I-II. Wiley, New York, 1988, 2592 pp.
	
	\bibitem{Edvards}\textbf{R.E. Edvards,} Functional Analysis: Theory and Applications, Dover Publisher, 783 pp.
	
	\bibitem{GogoladzeTsagareishvili1} \textbf{L. Gogoladze and G. Cagareishvili,} General Fourier coefficients and convergence almost everywhere, Izv. Ross. Akad. Nauk. Ser. Math.85 (2021), no.2, 60-72 and Izv. Math. 85 (2021), no.2, 228-240.
	
	\bibitem{GogoladzeTsagareishvili2} \textbf{L. Gogoladze and G. Cagareishvili,} On the absolute convergence of general Fourier series, Publ. Math. Debrecen, 100(2022), no.3-4, 277-294. 
	
	\bibitem{GogoladzeTsagareishvili3} \textbf{L. Gogoladze and G. Tsagareishvili,} Optimal convergence factors for general Fourier coefficients, Contemp. Math. Anal. (to appear).
		
	\bibitem{GogoladzeTsagareishvili} \textbf{L. Gogoladze and V. Tsagareishvili,} Fourier coefficients of continuous functions, Math. Notes 91 (2012), no.5-6, 645-656.
	
	
	\bibitem{Tsagareishvili1}	\textbf{L.Gogoladze and V.Tsagareishvili.} Differentiable functions and general orthonormal systems, Mosc. Math.J. 19 (2019), no.4, 695-707
	
	\bibitem{KashinSaakyan}\textbf{B.S. Kashin and A.A. Saakyan,} Orthogonal Series, Translations of Mathematical Monographs, American Mathematical Society, 2005, 451 pp. 
	
	\bibitem{Kantorovich}\textbf{L.V. Kantorovich and G.P. Akilov,} Functional Analysis, Pergamon Press, Oxford, 1982, 604 pp.
		
	\bibitem{KolmogorovFomin}\textbf{A.H. Kolmogorov and C.B.Fomin,} Elements The Theory of Functions, Martin Fine Books, 2012, 280 pp.
				
	\bibitem{Mclaughlin}\textbf{J.M. Mclaughlin,} Integrated orthonormal series, Pacific J. Math. 42 (1972), 469-475.
	
	\bibitem{Olevskii}\textbf{A. M. Olevskii,} Divergent Fourier series, Izv. Akad. Nauk SSSR Ser. Mat. 27 (1963), 343-366.	

	\bibitem{PTT} \textbf{L. E. Persson, G. Tephnadze and G. Tutberidze,} On the boundedness of subsequences of Vilenkin-Fejér means on the martingale Hardy spaces, Oper. Matrices, 14 (2020), no.1, 283-294.
	
	\bibitem{PTW} \textbf{L. E. Persson, G. Tephnadze and P. Wall,} On an approximation of 2-dimensional Walsh-Fourier series in the martingale Hardy spaces, Ann. Funct. Anal. 9 (2018), no.1, 137-150.
	
	\bibitem{PTW1}\textbf{L. E. Persson, G. Tephnadze and F. Weisz,} Martingale Hardy Spaces and Summability of one-dimensional Vilenkin-Fourier Series, book manuscript, Birkhäuser/Springer, 2022, 626 pp.
	
	\bibitem{PSTW} \textbf{L.-E. Persson, F. Schipp, G. Tephnadze and F. Weisz, }An analogy of the Carleson-Hunt theorem with respect to Vilenkin systems, J. Fourier Anal. Appl. 28 48 (2022),  1-29.
	
	\bibitem{tep1} \textbf{G. Tephnadze,} On the partial sums of Vilenkin-Fourier series, J. Contemp. Math. Anal. 54 (2019), no.6, 23-32.
	
	\bibitem{tep2} \textbf{G. Tephnadze,} On the partial sums of Walsh-Fourier series, Colloq. Math. 141 (2015), no.2, 227-242.
	
	
	\bibitem{tep4} \textbf{G. Tephnadze,} On the convergence of partial sums with respect to Vilenkin system on the martingale Hardy spaces, J. Contemp. Math. Anal. 53 (2018), no.5, 294-306.
	
		
	\bibitem{Tsagareishvili2} \textbf{V.Tsagareishvili,}
	Some properties of general orthonormal systems, Colloq. Math. 162 (2020), no.2, 201-209.
	
	\bibitem{Tsagareishvili3} \textbf{V.Tsagareishvili,} Some particular properties of general orthonormal systems, Period. Math. Hungar. 81 (2020), no.1, 149-157.
	
	\bibitem{Tsagareishvili4} \textbf{V.Tsagareishvili,}
	General orthonormal systems and absolute convergence, Izv. Ros. Acad. Nauk. Ser. Math. 84 (2020), no.4, 208-220 and Izv. Math. 84 (2020), no. 4, 816-828. 
	
	\bibitem{Tsagareishvili0} \textbf{V.Tsagareishvili.} On the majorants of general Fourier coefficients and best approximations, Publ. Math. Debrecen 102 (2023), no.3-4,  475-494.
	
	\bibitem{tsatut1} \textbf{V. Tsagareishvili and G. Tutberidze,} Multipliers of absolute convergence, Mat. Zametki 105 (2019), no.3, 433-443 and Mat. Notes 107 (2019), no.3-4, 439-448.
	
	\bibitem{tsatut2} \textbf{V. Tsagareishvili and G. Tutberidze,} Absolute convergence factors of Lipshitz class functions for general Fourier series, Georgian Math. J. 29 (2022), no.2, 309-315.
	
	\bibitem{tsatut3} \textbf{V. Tsagareishvili and G. Tutberidze,} Some problems of convergence of general Fourier series, Izv. Math. Acad. Nauk Armenii Math. 57 (2022), no.6, 70-80. 
	
	\bibitem{Yosida} \textbf{K. Yosida ,} Functional Analysis, Springer verlag, Berlin, 1995, 501 pp.

\end{thebibliography}
\end{document}